\documentclass[a4paper,11pt]{article}
\usepackage[latin1]{inputenc}
\usepackage{amsfonts}
\usepackage{amsmath}
\usepackage{latexsym}
\usepackage[all]{xy}
\usepackage{amsthm}
\usepackage{graphicx}
\usepackage[T1]{fontenc}
\usepackage{float}

\usepackage{hyperref}

\usepackage{caption2}

\usepackage{amscd}
\author{Álvaro Martínez-Pérez and
Manuel A. Morón \footnote{The authors are partially supported by
MTM 2006-00825}}

\date{Departamento de Geometría y Topología, Universidad Complutense de Madrid. Madrid 28040, Spain\\
        e-mail: mamoron@mat.ucm.es; \ alvaro\_martinez@mat.ucm.es}

\title{Inverse sequences, rooted trees and their end spaces}

\begin{document}
\maketitle

\newtheorem{definicion}{Definition}[section]
\newtheorem{nota}[definicion]{Remark}
\newtheorem{prop}[definicion]{Proposition}
\newtheorem{lema}[definicion]{Lemma}
\newtheorem{obs}[definicion]{Remark}
\newtheorem{teorema}[definicion]{Theorem}
\newtheorem{cor}[definicion]{Corollary}
\newtheorem{ejp}[definicion]{Example}
\newtheorem{contejp}[definicion]{Counterexample}

\newtheorem{definicion2}{Definition}[subsection]
\newtheorem{nota2}[definicion2]{Remark}
\newtheorem{prop2}[definicion2]{Proposition}
\newtheorem{lema2}[definicion2]{Lemma}
\newtheorem{obs2}[definicion2]{Remark}
\newtheorem{teorema2}[definicion2]{Theorem}
\newtheorem{cor2}[definicion2]{Corollary}
\newtheorem{ejp2}[definicion2]{Example}

\begin{abstract} In this paper we  prove that if we consider the standard real metric on
simplicial rooted trees then the category Tower-Set of inverse
sequences can be described by means of the bounded coarse geometry
of the naturally associated trees. Using this we give a
geometrical characterization of Mittag-Leffler property in inverse
sequences  in terms of the metrically proper homotopy type of the
corresponding tree and its maximal geodesically complete subtree.
We also obtain some consequences in shape theory. In particular we
describe some new representations of shape morphisms related to
infinite branches in trees.
\end{abstract}



\section{Introduction}
It has been proved the efficiency of the use of category theory
and categorical language to study more concrete mathematical
structures. Moreover the construction of functors between
categories allows us to translate  specific facts in an specific
framework to a  different one. An example of all above  is
Algebraic Topology, created by means of Topology jointly with
different functors to algebraic categories. Taking one step up on
abstraction, new categories are created from old ones to produce
new useful framework such as pro-categories (inverse systems) with
the full subcategories of Towers (inverse sequences) or
in-categories (directed systems) with the subcategories of
directed sequences.

Many developments in mathematics use the abstract algebraic
construction of pro-category to unify concepts, results and
procedures. For example, pro-categories are used to describe shape
theory in order to extend efficiently the algebraic treatment of
CW-complexes or polyhedra to more general classes with not so good
local properties. See \cite{Dydak-Segal}, \cite{MS1} and
\cite{Cor-Por}.

However the above mentioned  categorical, or even pro-categorical,
chain of constructions can have some not so good secondary effects
such as to convert the language itself in a new matter to learn.

One of the aims of this paper is to convert the category
\textbf{Tower-Set} into a geometrical language involving trees and
Coarse Geometry, giving so a new relation for Shape Theory. In
particular we relate it to Coarse Geometry of simplicial
$\mathbb{R}$-trees. In fact we do something more going  further in
the following Serre's observation, \cite{Ser} pages 18-19:
\emph{"...We therefore have an equivalence between pointed trees
and inverse systems of sets indexed by integers $\geq 1$"}. In
this phrase Serre was referring to simplicial trees.  Our purpose
is to  describe in a geometrical way the abstract language of
pro-categories, at least for inverse sequences and maps between
them, using trees and certain continuous maps between them.

We then prove that if we consider the standard real metric on
simplicial trees then the category of Towers can be described by
means of a homotopy relation akin to  the bounded Coarse Geometry
of the corresponding tree. See \cite{Roe1}, \cite{Roe2} for
anything herein related to  \emph{Coarse Geometry}.

Based on the above equivalence it is natural to ask for describing
results in one of the categories in terms of the other. This is
the case of the important  Mittag-Leffler  property for Towers.
The Mittag-Leffler property was considered by Grothendieck,
\cite{Groth}, in the realm of Algebraic Geometry. After the
inverse systems description of shape theory by Marde\v{s}i\'c and
Segal in \cite{MS2}, it became clear soon the relevance of this
property in shape theory. In fact this is a shape property in
nature because it is equivalent in pro-Set to the notion of
movability , see \cite{MS1}, introduced by Borsuk. Of special
relevance is the case of pro-Group. As one can see in
\cite{Dydak-Segal} Chapter VI, Mittag-Leffler property appears at
first in  the study of algebraic properties associated to shape
theory. This is because, in general,  information may be lost when
passing  from pro-categories to their limits as it is the case in
shape theory. However in the presence of Mittag-Leffler property
all this information is retained.

Our geometrical characterization of Mittag-Leffler property in
inverse sequences is given in terms of the metrically proper
homotopy type of the corresponding tree and its maximal
geodesically complete subtree.

We also reinterpret and reprove, from our context, some of the
basic properties of inverse sequences, some of them for inverse
sequences of groups. In particular the level morphisms convert to
simplicial maps between trees and the description of any morphisms
by a level one  is nothing  more than an approximation result by
simplicial maps. We do this following Marde\v{s}i\'c and Segal
text \cite{MS1}.

In \cite{M-M} the authors constructed an isomorphism of categories
involving real trees and ultrametric spaces. As described there,
it was mainly related to a paper due to Hughes \cite{Hug} but also
to \cite{M-P}. Anyway in \cite{M-M} we didn't treat anything
related to shape theory as did in \cite{M-P}.

In this paper, as a consequence of our construction, we are going
to get also some applications in shape theory. In fact we recover
some of the results obtained in \cite{M-P}, related to the
construction of ultrametrics (the main properties of this type of
metrics are demonstrated and beautifully exposed in \cite{Ro}) in
the sets of shape morphisms, by passing to the end, to infinity,
in our construction.

So, as a summary, we go further on Serre's observation converting
morphisms between inverse sequences into non-expansive metrically
proper homotopy classes of non-expansive maps between trees. Thus,
we represent the categorical framework of inverse sequences inside
the core of the bounded Coarse Geometry of trees. As a consequence
we obtain some basic constructions from \cite{MS1} and \cite{M-P}
related to Shape Theory.

Although our main source of information on $\mathbb{R}$--trees is
Hughes's paper \cite{Hug}, it must be also recommended the
classical book \cite{Ser} of Serre and the survey \cite{Be} of
Bestvina to go further. Let us say that in \cite{Morgan}, J.
Morgan treats a generalization of $\mathbb{R}$--trees called
$\Lambda$--trees. Moreover, Noncommutative Geometry is used, by
Hughes in \cite{H} , to study the local geometry of ultrametric
spaces and the geometry of trees at infinity

A notational convention is in order. We use
\textbf{Tower-$\mathcal{C}$} to denote the subcategory of
\textbf{pro-$\mathcal{C}$} whose objects are inverse sequences.

\section{Preliminaries.}

In \cite{M-M}, we proved an equivalence of categories between $\mathbb{R}$--trees
and ultrametric spaces which generalizes classical results of Freudenthal
ends for locally finite simplicial trees, see \cite{BQ}. Some results and
most of the language of that paper will be used here. We include in this
section the basic definitions from \cite{Hug} and \cite{M-M} and we
summarize without proofs some results which are relevant to this paper.

\begin{definicion} A \emph{real tree}, or \emph{$\mathbb{R}$--tree} is a metric space $(T,d)$
that is uniquely arcwise connected and $\forall x, y \in T$, the
unique arc from $x$ to $y$, denoted $[x,y]$, is isometric to the
subinterval $[0,d(x,y)]$ of $\mathbb{R}$.
\end{definicion}

\begin{definicion} A \emph{rooted $\mathbb{R}$--tree}, $(T,v)$ is an
$\mathbb{R}$--tree $(T,d)$ and a point $v\in T$ called \emph{the
root}.
\end{definicion}

\begin{definicion}\label{extensiongeod} A rooted $\mathbb{R}$--tree is
\emph{geodesically complete} if every isometric embedding
$f:[0,t]\rightarrow T,\ t>0$, with $f(0)=v$, extends to an
isometric embedding $\tilde{f}:[0,\infty) \rightarrow T$. In that
case we say that $[v,f(t)]$ can be extended to a \emph{geodesic
ray}.
\end{definicion}

\begin{definicion} If $c$ is any point of the rooted $\mathbb{R}$--tree
$(T,v)$, the \emph{subtree of $(T,v)$ determined by c} is:
\[T_c=\{x\in T | \ c\in [v,x]\}.\]
\end{definicion}

\begin{definicion}
A map $f$ between two metric spaces $X, \ X'$ is \emph{metrically
proper} if for any bounded set $A$ in $X'$,  $f^{-1}(A)$ is bounded
in $X$.
\end{definicion}

\begin{definicion} If $(X,d)$ is a metric space and \ $d(x,y)\leq \max \{d(x,z),d(z,y)\}$
for all $x,y,z\in X$, then $d$ is an \emph{ultrametric} and
$(X,d)$ is an \emph{ultrametric space}.
\end{definicion}

There is a classical relation between trees and ultrametric
spaces. The functors between the objects are defined as follows in \cite{Hug}.

\begin{definicion}\label{end} The \emph{end space} of a rooted
$\mathbb{R}$--tree $(T,v)$ is given by: \vspace{0.5cm}

\mbox{$end(T,v)=\{f:[0,\infty) \rightarrow T \ |\ f(0)=v$ and $f$
is an isometric embedding $\}.$} \vspace{0.5cm}

For $f,g\in end(T,v)$, define:

\[ d_e(f,g)= \left\{ \begin{tabular}{l} $0 \qquad  \mbox{ if }\
f=g,$\\

$e^{-t_0} \quad \mbox{ if } f\ne g \mbox{ and } t_0=sup\{t\geq 0 |
\ f(t)=g(t)\}$\end{tabular}
 \right.
\]
\end{definicion}

\vspace{0.5cm}

\begin{nota} Abusing of the notation, we sometimes identify
the element of the end space with its image on the tree. This will
be usually called \emph{branch}. Also, for non-geodesically
complete $\mathbb{R}$--trees, we also use \emph{branch} to call
any rooted non-extendable isometric embedding, making distinction
between finite and infinite branches.
\end{nota}

\begin{prop} For any point in a rooted $\mathbb{R}$--tree, $x\in (T,v)$, there
is a branch $F$ and some $t\in [0,\infty)$ such that \ $F(t)=x$.
\end{prop}

\begin{prop} If $(T,v)$ is a rooted $\mathbb{R}$--tree,
then $(end(T,v),d_e)$ is a complete ultrametric space of diameter
$\leq 1$.
\end{prop}

Let U be a complete ultrametric space with diameter $\leq 1$,
define:
\begin{displaymath}T_U :=\frac{U\times
[0,\infty)}{\sim}\end{displaymath}
with $(\alpha,t)\sim(\beta,t')\Leftrightarrow t=t' \quad $and$
\quad \alpha,\beta \in U \quad$ such that $\quad
d(\alpha,\beta)\leq e^{-t}.$ \vspace{0.5cm}

Given two points in $T_U$ represented by equivalence classes
$[x,t],[y,s]$ with $(x,t),(y,s)\in U\times [0,\infty)$ define a
metric on $T_U$ by:
\[D([x,t],[y,s])=\left\{
\begin{tabular}{l} $|t-s| \qquad \qquad \qquad \qquad \qquad
\qquad \mbox{ if }
x=y,$\\

$t+s-2\min\{-ln(d(x,y)),t,s\} \quad \mbox{ if } x\ne
y.$\end{tabular}
 \right.\]

\begin{prop}\label{induced tree} $(T_U,D)$ is a geodesically complete rooted $\mathbb{R}$-tree.
\end{prop}

Some of these tools can be adapted for the more general case of rooted $\mathbb{R}$--trees (not necessarily geodesically complete) using the fact that for any rooted $\mathbb{R}$--tree, $(T,v)$, there exists a unique geodesically
complete subtree, $(T_\infty,v)\subset (T,v)$, that is maximal.

\begin{lema}\label{retracto} If the metric of $(T_\infty,v)$ is proper then it
is a deformation retract of $(T,v)$.
\end{lema}

Of course in the framework of simplicial trees the subtree is
always a deformation retract but this is not true, in general, for
$\mathbb{R}$--trees.

\begin{ejp} Consider the following $\mathbb{R}$--tree $(T,v)$.
\end{ejp}

\begin{figure}[H]
\centering \scalebox{0.6}{\includegraphics{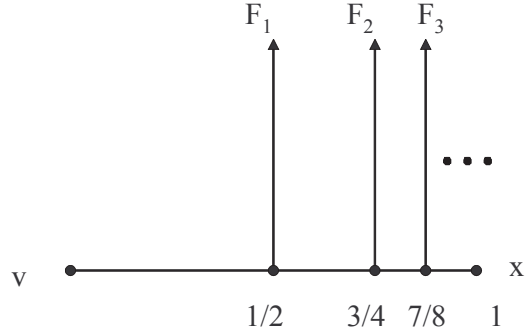}}
\caption{The maximal geodesically subtree is not a retract.}
\end{figure}

$(T,v)$ has a finite branch, $F_0$, of length 1 (from the root to
$x$), and geodesically complete branches $F_i$ bifurcating from
$F_0$ at a distance $\frac{2^i-1}{2^i}$ from the root.

The geodesically complete subtree $(T_\infty,v)$ is
$(T,v)\backslash \{x\}$. Clearly, in this case $(T_\infty,v)$, is
not a retract of $(T,v)$.

\begin{prop}\label{imagensubarboles}
Let $f:(T,v) \to (T',w)$ be a rooted continuous and metrically
proper map, and let $M>0$ and $N>0$ such that \quad
$f^{-1}(B(w,M)) \subset B(v,N)$, then
\[\forall c\in \ \partial B(v,N) \ \exists ! \ c' \in
\partial B(w,M) \ \mbox{such that} \quad f(T_c)\subset T'_{c'}.\]
\end{prop}

\begin{definicion}\label{short_homot} If $f,g:X\to T$ are two continuous maps from any
topological space $X$ to a tree $T$ then the \emph{shortest path
homotopy} is a homotopy $H:X\times I \to T$ of $f$ to $g$ such
that if $j_x:[0,d(f(x),g(x))]\to [f(x),g(x)]$ is the isometric
immersion of the subinterval $[0,d(f(x),g(x))]\subset \mathbb{R}$
into $T$ whose image is the shortest path between $f(x)$ and
$g(x)$, then $H(x,t)=j_x(t\cdot d(f(x),g(x))) \ \forall t \in I \
\forall x\in X$.
\end{definicion}

\begin{definicion} Given $f,f':(T,v) \to (T',w)$ two rooted continuous metrically proper
maps, let $H$ be a continuous map $H: T\times I \to T'$ with
\mbox{$H(v,t)=w$} \ $\forall t \in I$ such that $\forall M>0,
\exists N>0$ \ such that $H^{-1}(B(v,M))\subset B(v,N)\times I$.
Then, $H$ is a \emph{rooted metrically proper homotopy} of $f$ to
$f'$ if \quad $H|_{T\times \{0\}}=f$ and $H|_{T\times \{1\}}=f'$.
\end{definicion}

\textbf{Notation:} $f\simeq_{Mp} f'$ if and only if there exists a
rooted metrically proper homotopy of $f$ to $f'$.

\textbf{Notation:} We will denote $f \simeq_{L} f'$, rooted
metrically proper non-expansive homotopic, if there is a rooted
metrically proper homotopy of $f$ to $g$ which is non-expansive at
each level.

\textbf{Notation:} We will denote $f \simeq_{C} f'$, rooted coarse
homotopic, if there is a rooted metrically proper homotopy of $f$
to $g$ which is coarse at each level.

Consider the categories,

\vspace{0.5cm}

$\mathcal{T}$: Geodesically complete rooted $\mathbb{R}$-trees and
rooted metrically proper homotopy classes of rooted continuous
metrically proper maps.

\vspace{0.5cm}

$\mathcal{U}$: Complete ultrametric spaces of diameter $\leq 1$
and uniformly continuous maps. \vspace{0.5cm}

Our main results in \cite{M-M} are the following:

\begin{teorema}\label{tma equiv} There is
an equivalence of categories between $\mathcal{T}$ and
$\mathcal{U}$.
\end{teorema}

\begin{cor}\label{cor_lip} There is an equivalence of categories between $\mathcal{U}$
and the category of geodesically complete rooted
$\mathbb{R}$-trees with rooted metrically proper non-expansive
homotopy classes of rooted metrically proper non-expansive maps.
\end{cor}

\begin{cor}\label{cor_coar} There is an equivalence of categories between $\mathcal{U}$
and the category of geodesically complete rooted
$\mathbb{R}$-trees with rooted coarse homotopy classes of rooted
continuous coarse maps.
\end{cor}


\section{Inverse sequences}

In \cite{Ser}, Serre gives a description of some correspondence
between inverse sequences and simplicial trees. Here we extend
this relation to some categorial equivalences, considering the
usual morphism between inverse sequences after \cite{MS1}.

\begin{definicion} An inverse sequence $\underline{X}=(X_n,p_n)$ in the category
$\mathcal{C}$ is an inverse system in $\mathcal{C}$ indexed by the
natural numbers.
\end{definicion}

Let us denote $p_{nm}:X_m\to X_n$ the composition $p_n \circ
\cdots \circ p_{m-1}$.

\begin{definicion} A morphism of inverse sequences $(f_n,\Phi):(X_n,p_n)\to
(Y_n,q_n)$ consists of a function $\Phi:\mathbb{N}\to \mathbb{N}$
and morphisms $f_n:X_{\Phi(n)}\to Y_n$ in $\mathcal{C}$ such that
$\forall n'>n$ there exists $m\geq n,n'$ for which $f_n\circ
p_{\Phi(n)m}=q_{nn'}\circ f_{n'}\circ p_{\Phi(n')m}$.
\end{definicion}

There is an equivalence relation $\sim$ between morphisms of
sequences. We say that $(f_n,\Phi)\sim (g_n,\Psi)$ if every $n$
admits some $m\geq \Phi(n),\Psi(n)$ such that $f_n\circ
p_{\Phi(n)m}=g_{n}\circ p_{\Psi(n)m}$.

Let \textbf{Tower-$\mathcal{C}$} be the category whose objects are
inverse sequences in the category \textbf{$\mathcal{C}$} and whose
morphisms are equivalence classes of morphisms of sequences. The
particular case we are mostly going to treat is
\textbf{Tower-Set}, whose objects are inverse sequences in
\textbf{Set}, the category of small sets.

\subsection{Inverse sequence of a tree}

Let $(\Gamma,v)$ a rooted simplicial tree. For each integer $n\geq
0$ let $C_n$ be the set of vertices of $\Gamma$ such that the
distance to the root is $n$. For each vertex $P$ of $C_n$ there is
a unique adjacent vertex $P'$ distant $n-1$ to the root. This
defines a map $f_n:P \to P'$ of $C_n$ to $C_{n-1}$ and hence an
inverse sequence
\[C_1 \leftarrow C_2 \leftarrow \cdots \leftarrow C_n \leftarrow \cdots \]
Furthermore, every inverse sequence can be obtained this way.

\subsection{Tree of an inverse sequence}

Let $\underline{X}=(X_n,p_n,\mathbb{N})$ be an inverse sequence
(an inverse system with directed set $\mathbb{N}$). Consider the
union of the $X_n$ and an extra point $v$ the set of vertices of
$\Gamma_{\underline{X}}$ and the geometric edges are
$\{x_{n+1},p_n(x_{n+1})\}$ and $\{x_1,v\}$. Let
$T_{\underline{X}}=\mbox{real}(\Gamma_{\underline{X}})$ (assume
each edge with length 1), then $(T_{\underline{X}},v)$ is a rooted
simplicial tree. We therefore have an equivalence between rooted
simplicial trees and inverse sequences in \textbf{Set} category.

\section{Metrically proper maps and morphisms of inverse sequences}

\subsection{Metrically proper maps}

Let $f:(T,v)\to (T',w)$ be a rooted continuous metrically proper
map. We can induce from this map a morphism between inverse
sequences $(f_n,\Phi_f):(C_n,p_n,\mathbb{N})\to
(C'_n,p'_n,\mathbb{N})$.

Since $f$ is metrically proper, $\forall n \ \exists t_n\in
\mathbb{N}$ such that $f^{-1}(B(w,n))\subset B(v,t_n)$ and there
is no problem to assume $t_n>t_{n-1}$. Thus by
\ref{imagensubarboles}, $\forall c\in C_{t_n}$ there exists a
unique $c'\in C'_n$ such that $f(T_c)\subset T'_{c'}$. Then let
$\Phi_f(n)=t_n \ \forall n\in \mathbb{N}$ and $f_n(c)=c'$ defines
a map $f_n: C_{t_n}\to C'_n$. Obviously $p'_n\circ
f_{n+1}=f_n\circ p_{\Phi_f(n) \Phi_f(n+1)}$ and $(f_n,\Phi_f)$ is
a morphism of inverse sequences.

Another election of the $t_n$ would induce another morphism
$(f'_n,\Psi_f)$. It is immediate to see that in that case
$(f'_n,\Psi_f) \sim (f_n,\Phi_f)$. Suppose $t'_n=\Psi_f(n)\geq
t_n=\Phi_f(n)$, and let $d\in C_{t'_n}$, $c\in C_{t_n}$ with $c\in
[v,d]$ (hence $p_{t_n t'_n}(d)=c$), then there is a unique $c'\in
C'_n$ such that $f(T_d)\subset f(T_c)\subset T'_{c'}$ and clearly
$f_n\circ p_{t_n t'_n}=f'_n$. Hence, from a rooted continuous
metrically proper map $f$, we induce a unique class of morphisms
of inverse sequences $[\underline{f}]$, that is, a unique morphism
in \textbf{Tower-Set}.

\subsection{Morphisms between inverse sequences}\label{tilde}

Any morphism $(f_n,\Phi):\underline{X} \to \underline{Y}$ between
two inverse sequences induces a rooted continuous metrically
proper map between the rooted trees $(T_{\underline{X}},v)$ and
$(T_{\underline{Y}},w)$ of $\underline{X}$ and $\underline{Y}$. To
show this, first we need the following: An infinite branch of
$(T_{\underline{X}},v)$ is given by a sequence of vertices
$(x_n)_{n\in \mathbb{N}}$ with $x_n\in X_n$ and such that
$p_n(x_{n+1})=x_n \ \forall n$. A finite branch is given by a
finite sequence $(x_1,\cdots, x_m)$ such that $p_n(x_{n+1})=x_n \
\forall n<m$ and $x_m \not \in p_n(X_{m+1})$. The branches are the
realization of the graph formed by those vertices, the root $v$,
and the edges between them.

With this idea we can induce from the morphism $(f_n,\Phi)$ a
function which sends branches of $(T_{\underline{X}},v)$ to
branches of $(T_{\underline{Y}},w)$.

Given $(f_n,\Phi):\underline{X} \to \underline{Y}$ it is immediate
that,

$\exists t_1 > \Phi(1),\Phi(2)$ such that $f_1 \circ p_{\Phi(1)
t_1}=q_1\circ f_2 \circ p_{\Phi(2) t_1}$.

$\exists t_2 > t_1,\Phi(3)$ such that $f_2 \circ p_{\Phi(2)
t_2}=q_2\circ f_3 \circ p_{\Phi(3) t_2}$.

In general,
\begin{equation}\exists t_k > t_{k-1},\Phi(k+1) \emph{ such
that } f_i \circ p_{\Phi(i) t_k}=q_{ik+1}\circ f_{k+1} \circ
p_{\Phi(i+1) t_k} \ \forall i\leq k
\end{equation}

A sequence $(x_n)_{n\in \mathbb{N}}$ with $p_n(x_{n+1})=x_n$
(which represents a geodesically complete branch in
$(T_{\underline{X}},v)$) can be easily sent to
$(f_n(x_{\Phi(n)}))_{n \in \mathbb{N}}$. To see that this
represents a geodesically complete branch in
$(T_{\underline{Y}},w)$ it suffices to check that
$f_n(x_{\Phi(n)})=q_n(f_{n+1}(x_{\Phi(n+1)})) \ \forall n\in
\mathbb{N}$ and by definition of $(x_n)_{n\in \mathbb{N}}$ and
$t_n$, (1), we know that $f_n(x_{\Phi(n)})=f_n \circ p_{\Phi(n)
t_n}(x_{t_n})=q_n\circ f_{n+1} \circ p_{\Phi(n+1)
t_n}(x_{t_n})=q_n(f_{n+1}(x_{\Phi(n+1)}))$.

With the finite branches we have to be a little more careful. Let
$(x_1,\cdots, x_m)$ be the sequence of vertices associated to a
finite branch ($x_i=p_{im}(x_m)$). Let $k_0:=
\underset{t_k<m}{max}\{k\}$. Then, we can give another sequence in
the image tree
$(f_1(x_{\Phi(1)}),\cdots,f_{k_0+1}(x_{\Phi(k_0+1)}))$ which is
part of a branch of $(T_{\underline{Y}},w)$ since $t_{k_0}$ is
such that $f_i(x_{\Phi(i)})=f_i \circ p_{\Phi(i)
n}(x_n)=q_{ik_0+1}\circ f_{k_0+1} \circ p_{\Phi(i+1)
n}(x_n)=q_{ik_0+1}\circ f_{k_0+1}(x_{\Phi(k_0+1)}) \ \forall i\leq
k_0$.

Thus, for every branch $F$ of $(T_{\underline{X}},v)$ given by a
finite (or infinite) sequence of vertices $(x_i)_{i=1}^m$ (or
$(x_n)_{n\in \mathbb{N}}$), there is some branch $G$ in
$(T_{\underline{Y}},w)$ which contains the vertices
$(f_i(x_{\Phi(i)}))_{i=1}^{k_0+1}$ ($(f_n(x_{\Phi(n)}))_{n \in
\mathbb{N}}$), in particular, if $F$ is geodesically complete so
is $G$. Hence, from $(f_n,\Phi)$ we can induce this way a function
$\tilde{f}$ sending branches of $(T_{\underline{X}},v)$ to
branches of $(T_{\underline{Y}},w)$. Finally, let
$\hat{f}:(T_{\underline{X}},v) \to (T_{\underline{Y}},w)$ such
that if $t\leq t_1$ then $\hat{f}(F(t))=w$ and if $t\in
[t_k,t_{k+1}]$ then
$\hat{f}(F(t))=\tilde{f}(F)(k-1+\frac{t-t_k}{t_{k+1}-t_k})$ for
any branch $F$ of $(T_{\underline{X}},v)$. Let us see that this
map is well defined, rooted, continuous and metrically proper.

\underline{Well defined}. Consider a point of the tree with two
representatives $F(t)=G(t)$ and suppose $t\in [t_k,t_{k+1}]$.
Hence the image will be
$\tilde{f}(F)(k-1+\frac{t-t_k}{t_{k+1}-t_k})$ or
$\tilde{f}(G)(k-1+\frac{t-t_k}{t_{k+1}-t_k})$ but since $F\equiv
G$ on $[0,t_k]$, $F(i)=G(i) \ \forall i\leq t_k$. Then
$\tilde{f}(F)(i)=\tilde{f}(G)(i) \ \forall i\leq k+1$ and
$\tilde{f}(F)\equiv \tilde{f}(G)$ on $[0,k+1]$ and thus, the image
is unique.

It is obviously rooted and continuous, and clearly,
$\hat{f}^{-1}(B(w,k))\subset B(v,t_{k+1})$, and then, metrically
proper.

It is clear that the election of $t_k$ may affect to the induced
map. From another sequence $(t'_k)_{k\in \mathbb{N}}$ in the same
conditions, we will induce another map $\hat{f}'$ between the
trees but if we consider $H$ the shortest path homotopy
(\ref{short_homot}) of $\hat{f}$ to $\hat{f}'$, since
$\hat{f}(F(t_k))=\hat{f'}(F(t'_k))=G(k-1)$, $H(T\backslash
B(v,max\{t_k,t'_k\}))\subset T'\backslash B(w,k-1)$ which is
equivalent to $H^{-1}(B(w,k-1))\subset B(v,max\{t_k,t'_k\})\times
I$. Hence, there is a metrically proper homotopy between the
induced maps $\hat{f}$, $\hat{f}'$ and from a morphism in
\textbf{Tower-Set} we induce a unique metrically proper homotopy
class $[\hat{f}]_{mp}$ of rooted continuous metrically proper maps
between the trees.

\begin{prop2}\label{lip} The map $\hat{f}$ is non-expasive (Lipschitz of constant 1).
\end{prop2}

\begin{proof} If $x,x'$ are in the same branch $x=F(t),x'=F(t')$
then it is clear that $d(x,x')\geq d(\hat{f}(x),\hat{f}(x'))$
since intervals with length $t_{n+1}-t_n\geq 1$ are sent linearly
to intervals of length 1.

If $x,x'$ are not in the same branch $x=F(t),y=G(t')$ then let
$t_0=sup \{t| F(t)= G(t)\}$ and $y=F(t_0)=G(t_0)$.
$d(x,x')=d(x,y)+d(y,x')\geq
d(\hat{f}(x),\hat{f}(y))+d(\hat{f}(y),\hat{f}(x'))\geq
d(\hat{f}(x),\hat{f}(x'))$. \end{proof}

Since $\hat{f}$ is metrically proper and non-expansive it is
obvious that

\begin{cor2} The map $\hat{f}$ is coarse.
\end{cor2}

\section{The functors}

Remember that \textbf{Tower-Set} is the category of inverse
sequences in \textbf{Set} category with equivalence classes of
morphisms of sequences, and let $\mathcal{T^*}$ be the category of
rooted simplicial trees and metrically proper homotopy classes of
metrically proper maps between trees.

\begin{definicion}\label{xi} Let $\xi:$ \textbf{Tower-Set} $\to \mathcal{T^*}$ be such that
$\xi(\underline{X})=T_{\underline{X}}$ for any inverse sequence
and $\xi(\underline{f})=[\hat{f}]_{mp}$ for any morphism of
sequences.
\end{definicion}

\begin{prop} $\xi$ is a functor.
\end{prop}

\begin{proof} $\xi$ is \underline{well defined}. If $\underline{f}\sim
\underline{g}$, then $\xi(\underline{f})\simeq_{mp}
\xi(\underline{g})$. Suppose $(f,\Phi)\sim (g,\Psi)$. Then
$\forall n \ \exists m_n>\Phi(n),\Psi(n)$ such that $f_n\circ
p_{\Phi(n)m_n} = g_n\circ p_{\Psi(n)m_n}$. We can assume
$m_n>t_n(\Phi),t'_n(\Psi),m_{n-1}$.

For any sequence $\underline{x}=(x_1,\cdots,x_{m_n})$ with
$p_{im_n}(x_{m_n})=x_i \ \forall i<m_n$, the sequences
$(f_1(x_{\Phi(1)}),\cdots,f_n(x_{\Phi(n)}))\subset
\underline{f}(\underline{x})$ and
$(g_1(x_{\Psi(1)}),\cdots,g_n(x_{\Psi(n)}))\subset
\underline{g}(\underline{x})$ are such that
$f_n(x_{\Phi(n)})=f_n(p_{\Phi(n)m_n}(x_{m_n}))=g_n(p_{\Psi(n)m_n}(x_{m_n}))=g_n(x_{\Psi(n)})$.
Hence, for any branch $F$ such that $F(i)=x_i \ \forall i\leq
m_n$, then
$\tilde{f}(F)(i)=f_i(x_{\Phi(i)})=g_i(x_{\Phi(i)})=\tilde{g}(F)(i)
\ \forall i\leq n$ and $\tilde{f}(F)\equiv \tilde{g}(F)$ on
$[0,n]$.

Thus, and since $\forall t>m_n \
\hat{f}(F(t)),\hat{g}(F(t))\subset T_{\underline{Y}}\backslash
B(w,n)$, if we consider the shortest path homotopy
$H:T_{\underline{X}}\times I \to T_{\underline{Y}}$ of $\hat{f}$
to $\hat{g}$, it is immediate to see that $\forall n\in
\mathbb{N}$ \ $H_t(T\backslash B(v,m_n))\subset T'\backslash
B(w,n)$ $\forall t$, which is equivalent to say that
$H_t^{-1}(B(w,n))\subset B(v,m_n)$ $\forall t$, and hence, $H$ is
a metrically proper homotopy.

\underline{$\xi(id_{\emph{\textbf{Tower-Set}}})=id_\mathcal{T^*}$}.
If we consider the representative of the identity which is a level
morphism and the identity at each level, the induced morphism
between the trees if we assume $t_k=\Phi(k+1)=k+1$ sends each
point $F(t)$, with $F$ any branch of $T^*$ and $t\leq 2$, to $w$
and $F(t)$ with $t>2$ to $F(t-2)$. Clearly, there is a metrically
proper homotopy of the identity to this contraction.

\underline{$\xi(g\circ f)=\xi(g)\circ \xi(f)$}. Let
$(f,\Phi):\underline{X}\to \underline{Y}$ and
$(g,\Psi):\underline{Y}\to \underline{Z}$ two morphisms between
inverse sequences. First consider $\aleph =\Phi \circ \Psi$ and
$h=g\circ f$. To construct $\xi(f)$ and $\xi(g)$ we define the
sequences $(s_n)_{n\in \mathbb{N}}$ and $(r_n)_{n\in \mathbb{N}}$
respectively, satisfying condition (1). For $\xi(g\circ f)$, we
define this sequence $(t_n)_{n\in \mathbb{N}}$ to be
$t_n=r_{s_{n+1}+1}$ (note that (1) would be satisfied in $(g\circ
f,\aleph)$ for any $t_n\geq r_{s_n}$). Then, any branch $F$ given
by a sequence of vertices $(x_1,\cdots,x_n)$ with $t_{k}\leq n
\leq t_{k+1}$ is sent to a branch $G$ whose $k+1$ first vertices
are $(w,h_1(x_{\aleph(1)}),\cdots,h_k(x_{\aleph(k)}))$, and if
$t\in [t_{k},t_{k+1}]$ then $\hat{h}(F(t))\in G[k-1,k]$. If we
consider $\xi(g)\circ \xi(f)$ then we can assume that the branch
$F$ is sent to the same branch $G$, note that the first $k+1$
vertices of $G$ are
$(w,g_1(f_{\Psi(1)}(x_{\Phi(\Psi(1))})),\cdots,g_k(f_{\Psi(k)}(x_{\Phi(\Psi(k))})))$,
and also $\forall t\in
[r_{s_{k+1}+1},r_{s_{k+2}+1}]=[t_k,t_{k+1}]$ then
$\hat{g}(\hat{f}(F(t)))\subset G[k-1,k]$. Hence, the induced map
$\hat{h}$ doesn't need to coincide exactly with $\hat{g}\circ
\hat{f}$, but both send intervals $[t_k,t_{k+1}]$ to intervals
$[k-1,k]$ and coincide on the vertices at levels $t_k$ all because
of the election of $(t_n)_{n\in \mathbb{N}}$. This obviously
implies the existence of a metrically proper homotopy between them
and thus $\xi(g\circ f)=\xi(g)\circ \xi(f)$. \end{proof}

\begin{definicion} Let $\eta:\mathcal{T^*} \to$ \textbf{Tower-Set} be such that for any
rooted tree $(T,v)$, $\eta(T,v)=(C_n,p_n,\mathbb{N})$ and for any
rooted continuous metrically proper map $f$,
$\eta(f)=\underline{f}$ the equivalence class of $(f_n,\Phi_f)$.
\end{definicion}

\begin{prop} $\eta$ is a functor.
\end{prop}

\begin{proof} $\eta$ is \underline{well defined}. If $f\simeq_{mp}f'$ then
$(f_n,\Phi_f)\sim (f'_n,\Phi_{f'})$. Let $H:T\times I \to T'$ be a
rooted metrically proper homotopy of $f$ to $f'$. Then $\forall n
\ \exists m$ such that $H^{-1}(B(w,n))\subset B(v,m)\times I$ and
clearly, $\forall k>m,\Phi_f(n),\Phi_{f'}(n)$ $f_n\circ
p_{\Phi_f(n)k}=f'_n\circ p_{\Phi_{f'}(n)k}$ and hence,
$(f_n,\Phi_f)\sim (f'_n,\Phi_{f'})$.

It is immediate to see that
\underline{$\eta(id_\mathcal{T^*})=id_{\emph{\textbf{Tower-Set}}}$}.

\underline{$\eta(g\circ f)=\eta(g)\circ \eta(f)$}. Consider
$f:(T,u)\to (T',v)$ and $g:(T',v)\to (T'',w)$. Let $(s_n)_{n\in
\mathbb{N}}$ be an increasing sequence of integers such that
$g^{-1}(B(w,n))\subset B(v,s_n)$. Let $(r_n)_{n\in \mathbb{N}}$ an
increasing sequence of integers such that $f^{-1}(B(v,n))\subset
B(v,r_n)$. We can now define the sequence $(t_n)_{n\in
\mathbb{N}}$ such that $(g \circ f)^{-1}B(w,n)\subset B(u,t_n)$ as
$t_n=r_{s_n}$. Hence $\Phi_{g\circ f}=\Phi_{g}\circ \Phi_{f}$ and
$(g\circ f)_n=g_{t_n}\circ f_n$ and thus $\eta(g\circ
f)=\eta(g)\circ \eta(f)$. \end{proof}

\section{Equivalence of categories}

Recall the following lemma in \cite{McL}:

\begin{lema} Let $S:A \rightarrow C$ be a functor between two categories.
$S$ is an equivalence of categories if and only if is $full,\
faithful$ and each object $c\in C$ is isomorphic to $S(a)$ for
some object $a\in A$.
\end{lema}

\begin{teorema}\label{equiv categ} $\eta$ is an equivalence of categories.
\end{teorema}

\begin{proof} \underline{$\eta$ is full}. Let $\underline{f}$ be a class of
morphisms in \textbf{Tower-Set}. Consider the representative
$(f_n,\Phi)$ such that $q_n\circ f_{n+1}=f_n\circ
p_{\Phi(n)\Phi(n+1)}$. This allows us, in the construction of
$\xi((f_n,\Phi))$, to assume $t_n=\Phi(n+1)$. Hence the map
$\hat{f}=\xi((f_n,\Phi))$ between the trees would be
$\hat{f}(F(t))=w$ if $t\leq \Phi(2)$ and
$\hat{f}(F(t))=\tilde{f}(F)(n-1+\frac{t-\Phi(n+1)}{\Phi(n+2)-\Phi(n+1)})$
if $t\in [\Phi(n+1),\Phi(n+2)]$, where $\tilde{f}$ is the induced
map between the branches as in \ref{tilde}. It suffices to check
that $\eta(\hat{f})=(f'_n,\Psi)\sim (f_n,\Phi)$. Clearly,
$\Psi(n)=\Phi(n+2)$ and if we assume in the construction of
$\eta(\hat{f})$ that $t'_n=\Psi(n)$, then $f'_n=q_{nn+2}\circ
f_{n+2}=f_n\circ p_{\Phi(n)\Phi(n+2)}$, and obviously,
$(f'_n,\Psi)\sim (f_n,\Phi)$.

\underline{$\eta$ is faithful}. If $\eta(f)\sim \eta(g)$ then
$f\simeq_{mp} g$. This is an immediate consequence if we see that
for any rooted continuous metrically proper map $f:(T,v)\to
(T',w)$, $\xi \circ \eta(f)\simeq_{mp} f$. $\xi \circ
\eta(f):=\hat{f}$ is a rooted continuous metrically proper map and
let $H$ be the shortest path homotopy of $f$ to $\hat{f}$. Let
$\eta(f):=(f_n,\Phi)$ where $\Phi(n)=t_n$ and $f_n:C_{t_n}\to
C'_{n}$ are defined as in section 4. If $\tilde{f}$ is the induced
map between the branches (which we can assume to be the same for
$f$ and $\hat{f}$ since for any branch $F$ of $(T,v)$,
$\hat{f}(F)\subset f(F)$), the map
$\xi(\eta(f))=\xi((f_n,\Phi))=\hat{f}$ sends $F(t)$ to $w$ if
$t\leq \Phi(2)$ and if $t\in [\Phi(n),\Phi(n+1)]$, with $n\geq 2$,
$\hat{f}(F(t))=\tilde{f}(F)(n-2+\frac{t-\Phi(n)}{\Phi(n+1)-\Phi(n)})$.
It is clear, because of the election of $t_n=\Phi(n)$, that also
$f(F(\Phi(n)))\subset T'_{\tilde{f}(F)(n-2)}$ for $t\in
[\Phi(n),\Phi(n+1)]$ and hence, the shortest path between
$f(F(t))$ and $\hat{f}(F(t))$ is contained in
$T'_{\tilde{f}(F)(n-2)}$. Then $H^{-1}(B(w,n-2))\subset
B(v,\Phi(n))\ \forall n\in \mathbb{N}$ and $H$ is metrically
proper.

Finally, for every inverse sequence
$\underline{X}=(X_n,p_n,\mathbb{N})$ it is immediate that
$T_{\underline{X}}$ is such that $C_n=\partial B(v,n)=X_n$ and
$\eta(T_{\underline{X}})=\underline{X}$. \end{proof}

By \ref{cor_lip} and \ref{cor_coar} (also see \ref{lip}) we obtain
the following corollaries.

\begin{cor} There is an equivalence of categories between \textbf{Tower-Set} and the
category of rooted simplicial trees with rooted metrically proper
non-expansive homotopy classes of rooted metrically proper
non-expansive maps.
\end{cor}

\begin{cor} There is an equivalence of categories between \textbf{Tower-Set} and the
category of rooted simplicial trees with rooted coarse homotopy
classes of rooted continuous coarse maps.
\end{cor}

\section{Mittag-Leffler property from the point of view of Serre's equivalence}

We give the definition of Mittag-Leffler property from \cite{MS1}
restricted to the particular case when the index set is
$\mathbb{N}$.

\begin{definicion} Let \textbf{X}$=(X_n,p_n,\mathbb{N})$ be an inverse
sequence in \textbf{Tower-$\mathcal{C}$}. We say that \textbf{X}
is Mittag-Leffler (ML) if $\forall n_0 \in \mathbb{N} \quad
\exists n_1>n_0$ such that $\forall n>n_1, \ p_{n_0}\circ \cdots
\circ p_{n-2}\circ p_{n-1}(X_n)=p_{n_0}\circ \cdots \circ
p_{n_1-2}\circ p_{n_1-1}(X_{n_1})$.
\end{definicion}

\begin{nota} Note that this definition doesn't depend on the
category $\mathcal{C}$. In fact \textbf{X} is (ML) if and only if
is (ML) as inverse sequence in \textbf{Tower-Set}.
\end{nota}

\begin{definicion} We say that $\alpha \in X_{n_0}$ is \emph{extendable} to
$n_1$ if there exist some $\beta \in X_{n_1}$ such that
$p_{n_0}\circ \cdots \circ p_{n_1-2}\circ p_{n_1-1} (\beta) =
\alpha$.
\end{definicion}

\begin{nota} In $T_{\underline{X}}$ this means that the path which connects
$\alpha$ with the root extends to a branch of length $n_1$ in the
tree. Note that this extended branch connects the root with an
element $\beta\in X_{n_1}$.
\end{nota}

The Mittag-Leffler property may be reformulated as follows:

\begin{definicion}\label{ML-re} The inverse sequence $(X_n,p_n,\mathbb{N})$
is (ML) if $\ \forall n_0 \ \exists n_1>n_0$ such that $\forall
\alpha \in X_{n_0}$ extendable to $n_1$, then $\alpha$ is
extendable to $n$ $\forall n>n_1$.
\end{definicion}

\begin{nota}\label{ML-arbol} In $T_{\underline{X}}$ this means that for each level
$n_0$ there exist some level $n_1$ such that for every $\alpha \in
X_{n_0}$ whose path connecting it to the root extends to a branch
of length $n_1$, then $\ \forall n>n_1$ that path can be extended
to some branch of length n.
\end{nota}

\begin{prop}\label{geod complete} Let $\underline{X}=(X_n,p_n,\mathbb{N})$ be an inverse sequence
and $T_{\underline{X}}$ the correspondent tree. If $\underline{X}$
is (ML), then for each level $n_0$, there is a level $n_1>n_0$
such that for any point $\alpha \in X_{n_0}$ extendable to $n_1$,
the path in $T_{\underline{X}}$ which connects the root with the
vertex $\alpha$ is geodesically complete.
\end{prop}

\begin{proof} (ML) means, see remark \ref{ML-arbol}, that for
each level $n_0$ there is a level $n_1>n_0$ such that for any
vertex $\alpha \in X_{n_0}$ extendable to $n_1$, the path of the
tree which connects the root with the vertex $\alpha$ extends to a
path of length n $\forall n>n_1$. To see that the path extends to
a geodesically complete branch of the tree we proceed by
induction. First we extend it to level $n_0+1$ this way.

Since the inverse sequence is (ML), we apply this property at
level $n_0+1$. Hence, there exist some $N_1>n_0+1$ such that any
$\beta\in X_{n_0+1}$ extendable to $N_1$ is extendable to $N$,
$\forall N>N_1$ (see definition \ref{ML-re}). There is no problem
to assume $N_1>n_1$. If we apply (ML) to level $n_0$, it is clear
that also $\alpha$ is extendable to $N$ $\forall N>N_1$. This
implies that there exist some $\gamma\in X_{N}$ such that
$p_{n_0}\circ \cdots \circ p_{N-2}\circ p_{N-1}(\gamma)=\alpha$,
and that $\alpha':=p_{n_0+1}\circ \cdots \circ p_{N-2}\circ
p_{N-1}(\gamma)\in X_{n_0+1}$ is extendable to $N_1$. This allows
us to repeat the induction argument, and hence, the path is
geodesically complete. \end{proof}

It is immediate to see the following:

\begin{obs} A tree is geodesically complete if and only if all the
bonding maps of the induced inverse sequence are surjective.
\end{obs}

Therefore, the maximal geodesically complete subtree is the
maximal subtree such that all the bonding maps of its inverse
sequence are surjective.

In \cite{MS1} we can find the following theorem at $[II.\S 6.2]$
referred to inverse systems.

\begin{prop} $\underline{X}$ is (ML) if and only if it is
isomorphic to an inverse sequence with surjective bonding maps.
\end{prop}

With this, and by theorem \ref{equiv categ} we can give the
following:

\begin{prop}\label{ML-prop} $\underline{X}$ is (ML) if and only if there is a rooted
metrically proper homotopy equivalence between $T_{\underline{X}}$
and its maximal geodesically complete subtree $T_\infty$. Moreover
the homotopy can be chosen to be a deformation retract.
\end{prop}

\begin{proof} Suppose $\underline{X}$ is (ML). By \ref{geod complete}, for each
level $n$, there is a level $t_n>n$ such that for any point
$\alpha \in X_{n}$ extendable to $t_n$, the path in
$T_{\underline{X}}$ which connects the root with the vertex
$\alpha$ is geodesically complete.

Let $T_\infty$ the maximal geodesically complete subtree. For each
point $x\in T_{\underline{X}}$ let $y_x\in T_\infty$ be such that
$d(x,T_\infty)=d(x,y_x)$ and $j_x:[0,d(x,T_\infty)]\to [x,y_x]$
the isometry from the subinterval in $\mathbb{R}$ to the unique
arc between $x$ and $y_x$. Thus, let $H:T_{\underline{X}}\times I
\to T_{\underline{X}}$ such that $H(x,t)=j_x(t\cdot
d(x,T_\infty))$. Clearly $H$ is an homotopy such that $H_0=id$ and
$H_1=r: T_{\underline{X}} \to T_\infty$ with $H(x,t)=x \ \forall
t\in I \ \forall x \in T_\infty$ ($T_\infty$ is a deformation
retract of $T_{\underline{X}}$, by \ref{retracto}, since the
metric of a simplicial tree is proper when we consider the edges
of length 1).

This homotopy $H$ is metrically proper. For every finite branch
$F$ with length $m\geq t_{n} $ there is a geodesically complete
branch extending the subbranch of length $n$ and hence the
homotopy $H$ sends the points on $T_{F(t_{n})}$ to $T_{F(n)}$ and
hence $H^{-1}(B(w,n))\subset B(v,t_{n})$.

Conversely, this equivalence implies that the inverse sequence is
isomorphic to the inverse sequence induced by the geodesically
complete subtree, whose bonding maps are obviously surjective.
\end{proof}



If we consider two inverse sequences to be related if and only if
they are isomorphic and the correspondent equivalence of maps as
Marde\v{s}i\'c and Segal do to define the shape category in
\cite{MS1} I $\S 2.3$ we get the following result:

\begin{prop}\label{ML-cat} There is an equivalence of categories between
classes of (ML) inverse sequences with classes of morphisms
between them and isomorphism classes of rooted (simplicial)
geodesically complete trees with classes of metrically proper
homotopy classes of rooted continuous metrically proper maps.
\end{prop}

The condition on the trees of being simplicial may be omitted by
the following proposition.

\begin{prop} For every rooted $\mathbb{R}$-tree $(T,v)$ there is a simplicial
rooted tree $(T',w)$ such that $(T,v)\simeq_L (T',w)$. Moreover
there is a bi-Lipschitz homeomorphism between $end(T,v)$ and
$end(T',w)$.
\end{prop}

\begin{proof} Let $(T,v)$ be an $\mathbb{R}$-tree. Let
$C_n:=\partial B(v,n)$ and $p_n: C_{n+1}\to C_n$ with
$p_n(c_{n+1})=c_n$ if and only if $c_n\in [v,c_{n+1}]$.
$\underline{C}=(C_n,p_n,\mathbb{N})$ is an inverse sequence. Let
$(T_{\underline{C}},w)$ the induced rooted simplicial tree. Then
there is a rooted metrically proper non-expansive homotopy
equivalence.

Let $f:(T_{\underline{C}},w) \to (T,v)$ be such that $f(w)=v$,
$f(c_n)=c_n$ and for each edge $f([c_n,c_{n+1}])=[c_n,c_{n+1}]$
the isometric embedding. The map $f$ is well defined, rooted,
continuous, metrically proper and non-expansive.

For each branch $F$ of $(T,v)$ there is a branch $\tilde{g}(F)$ on
$(T_{\underline{C}},v)$ whose vertices are $F(n)$ with $n\in
\mathbb{N}$ ($n=1,\cdots , k$ if $F$ is finite). To define
$g:(T,v) \to (T_{\underline{C}},w)$ let $g(\overline{B}(v,1))=w$
and $g(F(t))=\tilde{g}(F)(t-1)$ if $t>1$. The map $g$ is well
defined, rooted continuous metrically proper and non-expansive.

Both $g\circ f$ and $f\circ g$ send any point $F(t)$ to $F(t-1)$
if $t>1$. Hence both are rooted metrically proper non-expansive
homotopic to the identity (the shortest path homotopy is
non-expansive at each level).

If we consider the induced map $\tilde{f}: end(T',w)\to end(T,v)$
it is clearly a bijection. It is also immediate to see that
$\forall F,G \in end(T',w)$ there is some $n_0\in \mathbb{N}$ such
that $d(F,G)=e^{-n_0}$. This means that $F(n)=G(n)\quad \forall
n\leq n_0$ and $F(n_0+1)\neq G(n_0+1)$. It is clear from the
construction of $T'$ that $\tilde{f}(F)(n_0)=\tilde{f}(G)(n_0)$
and $\tilde{f}(F)(n_0+1)\neq \tilde{f}(G)(n_0+1)$. Hence
$e^{-n_0-1}<d(\tilde{f}(F),\tilde{f}(G))\leq e^{-n_0}$ and thus
$\frac{1}{e}d(F,G)< d(\tilde{f}(F),\tilde{f}(G))\leq d(F,G)$.
\end{proof}

This result, with \ref{tma equiv} yields

\begin{cor} For any complete ultrametric space of diameter $\leq
1$ $(X,d)$, there is a simplicial rooted tree $(T,v)$ such that
$end(T,v)$ is bi-Lipschitz homeomorphic to $(X,d)$.
\end{cor}

In particular, let us consider the category $\mathcal{U^*}$ whose
objects are uniformly homeomorphic classes of complete ultrametric
spaces of diameter $\leq 1$ and whose morphisms are classes of
uniformly continuous maps, where two uniformly continuous maps
$f,g$ are related if the following diagram commutes

\begin{displaymath}
\xymatrix{X \ar[r]^i \ar[d]_f & X' \ar[d]^{f'}\\
            Y \ar[r]_j & Y'}
\end{displaymath}

with i,j uniform homeomorphisms.

Similarly, let $\mathcal{S^*}$ be the category whose objects are
metrically proper homotopy classes of (ML) rooted simplicial trees
and whose morphisms are classes of morphisms in $\mathcal{T^*}$
making the diagram commutative

\begin{displaymath}
\xymatrix{S \ar[r]^i \ar[d]_f & S' \ar[d]^{f'}\\
          T \ar[r]_j & T'}
\end{displaymath}

with i,j rooted metrically proper homotopy equivalences.

Then, by \ref{ML-cat}, we can state:

\begin{prop} There is an equivalence of categories between
$\mathcal{U^*}$ and $\mathcal{S^*}$.
\end{prop}

Hence, if \textbf{Tower-Set$^*_{ML}$} is the category whose
objects are isomorphic classes of (ML) inverse sequences and whose
morphisms are classes of morphisms in \textbf{Tower-Set} where
$f\sim f'$ if the diagram commutes

\begin{displaymath}
\xymatrix{\underline{X} \ar[r]^i \ar[d]_{\underline{f}} & \underline{X'} \ar[d]^{\underline{f'}}\\
            \underline{Y} \ar[r]_j & \underline{Y'}}
\end{displaymath}

with i,j isomorphisms in \textbf{Tower-Set}.

\begin{cor} There is an equivalence of categories between
\textbf{Tower-Set$^*_{ML}$} and  $\mathcal{U^*}$
\end{cor}

\begin{cor} The shape morphisms in the sense of
Marde\v{s}i\'c-Segal between (ML) inverse sequences can be
represented by classes of uniformly continuous maps between
bounded ultrametric spaces.
\end{cor}

\section{Level morphisms and simplicial maps}

In the particular case of level morphisms between inverse
sequences we will see that we can induce a map between the trees
which is simplicial, preserves the distance from the root and is
in the same class of the map obtained with the functor $\xi$
defined in \ref{xi}.

\begin{definicion} $(f_n,\Phi): (X_n,p_n,\mathbb{N})\to
(Y_n,q_n,\mathbb{N})$ is a level morphism of sequences if
$\Phi:\mathbb{N}\to \mathbb{N}$ is the identity and $\forall n \in
\mathbb{N} \quad f_n\circ p_n=q_n\circ f_{n+1}$.
\end{definicion}

\begin{prop} A level morphism $(f_n,\Phi):\underline{X} \to
\underline{Y}$ induces a rooted simplicial map
$f:T_{\underline{X}}\to T_{\underline{Y}}$ which preserves the
distance to the root. Moreover this simplicial map is in the same
class of the metrically proper map induced between the trees by
the functor.
\end{prop}

\begin{proof} Let $f(v)=w$. Since $f_n:X_n \to Y_n$ send
vertices to vertices $\forall n\in \mathbb{N}$ and $\forall x_n\in
X_n$ let $f(x_n):=f_n(x_n)$. An edge in $T_{\underline{X}}$ is a
pair $[x_n,x_{n+1}]$ with $x_n\in X_n,x_{n+1}\in X_{n+1}$ and
$p_n(x_{n+1})=x_n$ and its image $f([x_n,x_{n+1}])$ will be
$[f_n(x_n),f_{n+1}(x_{n+1})]$ which is an edge in
$T_{\underline{Y}}$ since $q_n
(f_{n+1}(x_{n+1}))=f_n(p_n(x_{n+1}))=f_n(x_n)$.

To construct the metrically proper map $\xi(f)$ we can suppose
$t_n:=n+1$ and hence $\forall t\geq 2$, $\hat{f}$ sends
$F(t)=\tilde{f}(F)(t-2)$ and $\forall n\geq 2$
$\hat{f}(x_n)=q_{n-1}(q_n (f_n(x_n)))$. Thus, the equivalence
between the maps is obvious. \end{proof}

By \cite{MS1} I $\S 1.3$:

\begin{prop} Let $(f_n,\Phi):\underline{X} \to
\underline{Y}$ be any representant of any morphism in
\textbf{Tower-$\mathcal{C}$}. Then there exist inverse sequences
$\underline{X}'$ and $\underline{Y}'$, isomorphisms
$\textbf{i}:\underline{X} \to \underline{X}',
\textbf{j}:\underline{Y} \to \underline{Y}'$ in
\textbf{Tower-$\mathcal{C}$} and $(f'_n,id)$ a level morphism such
that $\textbf{j}\circ (f_n,\Phi)=(f'_n,id)\circ \textbf{i}:
\underline{X}\to \underline{Y}'$.
\end{prop}

Hence if we consider the category \textbf{Tower-Set$^*$} of
equivalence classes of isomorphic inverse sequences and the
correspondent classes of morphisms (see \cite{MS1}) then in every
class (in particular, for any shape morphism) there is a
representative which is a level morphism. Hence, in the equivalent
category of classes of simplicial rooted trees, in every class of
morphisms there is a simplicial map preserving the distance to the
root. Hence we can reduce this category to isomorphic classes of
simplicial rooted trees and classes of simplicial maps preserving
the distance to the root.

\begin{prop} There is an equivalence of categories between
\textbf{Tower-Set$^*$} and the category of isomorphic classes of
rooted simplicial trees with metrically proper homotopy classes of
simplicial maps preserving the distance to the root.
\end{prop}

\begin{nota} Any shape morphism in \textbf{Tower-Set} can be
represented by a simplicial map between rooted simplicial trees
preserving the distance to the root.
\end{nota}

\paragraph{Pro-groups}

In this section we study some classic results in pro-groups which
appear in \cite{MS1}, in terms of $\mathbb{R}$--trees. We obtain
alternative proofs, in geometric terms and in some case,
significantly different, of some of the results.

\begin{lema} If $(G_n,p_n)$ is an inverse sequence
in \textbf{Tower-Grp}, with \textbf{Grp} the category of groups
and homomorphisms, we consider the discrete topology at each
$G_n$, then $G=\underset{\leftarrow}{lim}(G_n)$ with the inverse
limit topology is a complete ultrametric topological group.
Moreover translations and inverse are isometries.
\end{lema}

\begin{proof} This inverse limit topology, the induced topology
as a subspace or $\underset{n\in \mathbb{N}}{\Pi} G_n$, if we
consider the discrete topology at each $G_n$ coincides with the
ultrametric topology as end space of the correspondent tree of the
inverse sequence in \textbf{Tower-Grp}.

In this inverse limit, translations and inverse are isometries.
Let $\underline{g}:=(g_n)_{n\in
\mathbb{N}},\underline{h}:=(h_n)_{n\in \mathbb{N}}\in G$ such that
$d(\underline{g},\underline{h})=e^{-n_0}$, this is $g_{n}=h_n
\forall n\leq n_0$ and $g_{n_0+1}\neq h_{n_0+1}$. Let
$\underline{k}:=(k_n)_{n\in \mathbb{N}}\in G$ and the translation
$G \to G$ given by $\underline{x}:=(x_n)_{n\in \mathbb{N}} \to
\underline{k}\cdot \underline{x}= (k_n\cdot x_n)_{n\in
\mathbb{N}}$. Clearly, $k_n\cdot g_n=k_n\cdot h_n \ \forall n\leq
n_0$ and $k_{n_0+1}\cdot g_{n_0+1} \neq k_{n_0+1}\cdot h_{n_0+1}$
and thus $d(\underline{k}\cdot \underline{g},\underline{k}\cdot
\underline{h})=e^{-n_0}$.

Similarly $g_n^{-1}=h_n^{-1} \ \forall n\leq n_0$ and
$g_{n_0+1}^{-1}\neq h_{n_0+1}^{-1}$ and hence
$d(\underline{g}^{-1},\underline{h}^{-1})=d(\underline{g},\underline{h})$.
\end{proof}

\begin{lema} If $(G_n,p_n)$ and $(H_n,q_n)$ are inverse sequences
in \textbf{Tower-Grp} with the discrete topology at each level,
$G=\underset{\leftarrow}{lim}(G_n)$ and
$H=\underset{\leftarrow}{lim}(H_n)$ with the inverse limit
topology and $f:G\to H$ is continuous then, $f$ is uniformly
continuous.
\end{lema}

\begin{proof} Since it is continuous at $0_G$, $\forall
\epsilon>0$ there exists $\delta >0$ such that $\forall g\in G$
with $d(g,0_G)<\delta$ then $d(f(g),0_H)<\epsilon$.

Let $h,h'\in G$ such that $d(h,h')<\delta$. Then, since
translations are isometries, $d(h'^{-1}\cdot h,0_G)<\delta$ and
hence $d(f(h'^{-1}\cdot h),0_H)<\epsilon$, and $d(f(h'^{-1}\cdot
h),0_H)=d(f(h')^{-1}\cdot f(h),0_H)=d(f(h),f(h'))<\epsilon$.
\end{proof}

By \ref{geod complete}

\begin{lema}\label{ML-pro-grupos} If $(G_n,p_n)$ is a (ML) inverse
sequence in \textbf{Tower-Grp},
$G=\underset{\leftarrow}{lim}(G_n)$ and $\pi_n:G\to G_n$ the
natural projection then every n admits some $m>n$ such that
$p_{nm}(G_m)=\pi_n(G)$.
\end{lema}

\begin{prop}\label{isomorf pro-grupos} If $(G_n,p_n)$ is a (ML) inverse
sequence in \textbf{Tower-Grp},
$G=\underset{\leftarrow}{lim}(G_n)$ and $\pi_n:G\to G_n$ the
natural projection then $(G_n,p_n)\approx (\pi_n(G),p_n|)$ are
isomorphic in \textbf{Tower-Grp}.
\end{prop}

\begin{proof} Let $i_n:\pi_n(G)\to G_n$ the natural inclusion,
which is obviously an homomorphism. $(i_n)$ is a level morphism in
\textbf{Tower-Grp}. To define $(f_n,\Phi): (G_n,p_n)\to
(\pi_n(G),p_n|)$ consider for each n the (ML) index $m>n$ and
define $\Phi(n)=m$, then by \ref{ML-pro-grupos}
$p_{nm}(G_m)=\pi_n(G)$ and hence we can define $f_n :=p_{nm}: G_m
\to \pi_n(G)$. It is clear that $(f_n)\circ (i_n)\sim id_{(G_n)}$
and $(i_n)\circ (f_n) \sim id_(\pi_n(G))$.\end{proof}

A morphism $f:X\to Y$ in an arbitrary category $\mathcal{C}$ is a
monomorphism provided $f\circ g=f\circ g'$ implies $g=g'$ for any
morphism $g,g':X'\to X$. Similarly, $f:X\to Y$ is an epimorphism
provided $g\circ f=g'\circ f$ implies $g=g'$ for any morphism
$g,g':Y\to Y'$. The following characterizations of monomorphism
and epimorphism of pro-groups are in \cite{MS1} and we adapt them
to the particular case of inverse sequences of groups.

\begin{lema} Let $\underline{G}=(G_n,p_n)$ and
$\underline{H}=(H_n,q_n)$ be inverse sequences of groups and let
$\underline{f}:\underline{G}\to \underline{H}$ a morphism in
\textbf{Tower-Grp} given by a level morphism $(f_n)$.
$\underline{f}$ is a monomorphism if and only if the following
condition holds:

(M) For every n there exists a $m\geq n$ such that
\[Ker(f_m)\subset Ker(p_{nm})\]
\end{lema}

\begin{lema} Let $\underline{G}=(G_n,p_n)$ and
$\underline{H}=(H_n,q_n)$ be inverse sequences of groups and let
$\underline{f}:\underline{G}\to \underline{H}$ a morphism in
\textbf{Tower-Grp} given by a level morphism $(f_n)$.
$\underline{f}$ is an epimorphism if and only if the following
condition holds:

(E) For every n there exists a $m\geq n$ such that
\[Im(q_{nm})\subset Im(f_{n})\]
\end{lema}

It is also proved in \cite{MS1} the following

\begin{prop}\label{monomorf y epimorf} Let $\underline{f}: \underline{G}\to \underline{H}$ a
morphism in \textbf{Tower-Grp}. $\underline{f}$ is an isomorphism
in \textbf{Tower-Grp} if and only if it is a monomorphism and an
epimorphism.
\end{prop}

\begin{prop}\label{monomorf} Let $(f_n):\underline{G}\to \underline{H}$ be a
level morphism of inverse sequences of groups which induces an
isomorphism of groups
$\tilde{f}:\underset{\leftarrow}{lim}(\underline{G})\to
\underset{\leftarrow}{lim}(\underline{H})$. If $\tilde{f}$ is open
and $\underline{G}$ has (ML) property, then the induced morphism
$\underline{f}:\underline{G}\to \underline{H}$ is a monomorphism
in \textbf{Tower-Grp}.
\end{prop}

\begin{proof} Since $\tilde{f}$ is a bijective open map then
$\forall \epsilon>0$ there exists $\delta>0$ such that if
$d(\tilde{f}(\underline{g}),0_H)<\delta$ then
$d(\underline{g},0_G)<\epsilon$. The metric in the inverse limit
is the ultrametric as end space of a tree. Thus, for
$\underline{g}=(g_n)_{n\in \mathbb{N}}\in G$,
$d(\underline{g},0_G)=e^{-sup\{n| g_n=0\}}$.

We want to check (M) for $(f_n)$.  For every $n_0$ let
$\epsilon=e^{-n_0}$. Let $\delta>0$ with the condition above and
consider $m_0>-ln(\delta)$. Since $(G_n,p_n)$ is (ML) consider
$m_1>m_0$ such that $p_{m_0m_1}(G_{m_1})=\pi_{m_0}(G)$ (see
\ref{ML-pro-grupos}).

If $x_{m_1}\in Ker(f_{m_1})$, $p_{m_0m_1}(x_{m_1})\in
Ker(f_{m_0})$ since $(f_n)$ is a level morphism and the diagram
commutes.

Let $\underline{g}=(g_n)_{n\in \mathbb{N}}\in G$ be such that
$\pi_{m_0}(\underline{g})=g_{m_0}=p_{m_0m_1}(x_{m_1})$. Since
$(f_n)$ is a level morphism $g_n\in Ker(f_n) \ \forall n\leq m_0$.
Then $f_n(g_n)=0 \ \forall n\leq m_0$ and
$d(\tilde{f}(\underline{g}),0_H)\leq e^{-m_0}<\delta$. Hence
$d(\underline{g},0_G)\leq \epsilon=e^{-n_0}$ which implies that
$g_n=0 \ \forall n\leq n_0$ where $0=g_{n_0}=p_{m_0m_1}(x_{m_1})$
and finally $x_{m_1}\in Ker(p_{m_0m_1})$.\end{proof}

\begin{prop}\label{epimorf} Let $(f_n):\underline{G}\to \underline{H}$ be a
level morphism of inverse sequences of groups such that the
induced morphism
$\tilde{f}:\underset{\leftarrow}{lim}(\underline{G})\to
\underset{\leftarrow}{lim}(\underline{H})$ is surjective. If
$\underline{H}$ has (ML) property, then the induced morphism
$\underline{f}:\underline{G}\to \underline{H}$ is an epimorphism
in \textbf{Tower-Grp}.
\end{prop}

\begin{proof} We need to check (E) for $(f_n)$.  Let $n_0\in
\mathbb{N}$. Since $\underline{H}$ is (ML) there is some $m_0>n_0$
such that $q_{n_0m_0}(H_{m_0})=\pi_{m_0}(H)$. If $y_{m_0}\in
H_{m_0}$ then $q_{n_0m_0}(y_{m_0})\in
q_{n_0m_0}(H_{m_0})=Im(q_{n_0m_0})$. Let
$\underline{h}=(h_n)_{n\in \mathbb{N}}\in H$ be such that
$\pi_{n_0}(\underline{h})=h_{n_0}=p_{n_0m_0}(y_{m_0})$. Since
$\tilde{f}$ is surjective there is some $\underline{g}=(g_n)_{n\in
\mathbb{N}}\in G$ such that
$\tilde{f}(\underline{g})=\underline{h}$ and this implies that
$f_n(g_n)=h_n=\pi_n(\underline{h}) \ \forall n$, and hence
$h_{n_0}=q_{n_0m_0}(y_{m_0})\subset
f_{n_0}(G_{n_0})=Im(f_{n_0})$.\end{proof}

We can recall the classical result.

\begin{prop}\label{homomorf abto} If $G$ and $H$ are separable and completely metrizable
topological groups and if $h:G\to H$ is a surjective continuous
homomorphism then h is open.
\end{prop}

\begin{lema}\label{separable pro-grupo} Let $(G_n,p_n)$ be an inverse sequence
in \textbf{Tower-Grp}. Then $G=\underset {\leftarrow}
{lim}(\underline{G})$ is separable if and only if $\forall n\in
\mathbb{N} \quad \pi_n(G)$ is countable (with $\pi_n:G\to G_n$ the
natural projection).
\end{lema}

\begin{proof} If $\pi_n(G)$ is countable and we consider for
each $n$ and each element $g_n\in \pi_n(G)$ an element $g\in G$
such that $\pi_n(g)=g_n$ we have a countable dense subset. If
there is some $n$ with $\pi_n(G)$ not countable, then
$\{\pi_n^{-1}(g_n)| \ g_n\in \pi_n(G)\}$ defines an uncountable
partition of $G$, and hence, G is not separable.\end{proof}

As a corollary of this we can give the following theorem which is
almost the same in \cite{MS1} (II,$\S 6.2$ Theorem 12) where it is
proved using an exact sequence and the first derived limit. Here
we present a slightly stronger version with a more direct and
geometrical proof.

\begin{teorema} Let $(f_n):\underline{G}\to \underline{H}$ be a
level morphism of inverse sequences of groups which induces an
isomorphism $\tilde{f}:\underset{\leftarrow}{lim}(G_n) \to
\underset{\leftarrow}{lim}(H_n)$. If $\underline{G}$ and
$\underline{H}$ have the (ML) property and all $\pi_n(G)$ are
countable, then the induced morphism $\underline{f}:
\underline{G}\to \underline{H}$ is an isomorphism in
\textbf{Tower-Grp}.
\end{teorema}

\begin{proof} Since $\tilde{f}$ is surjective $\pi_n (H)$ is
also countable, and by lemma \ref{separable pro-grupo} $G$ and $H$
are separable. Since $\tilde{f}$ is the induced map between the
limits by a level morphism, it can be considered as the induced
map between the end spaces by a metrically proper map between the
trees and hence it is uniformly continuous with the induced
ultrametric. Thus, by \ref{homomorf abto} it is open and by
propositions \ref{monomorf} and \ref{epimorf} the induced morphism
in \textbf{Tower-Grp} $\underline{f}$ is a monomorphism and an
epimorphism, and hence (see \ref{monomorf y epimorf})
$\underline{f}$ is an isomorphism in \textbf{Tower-Grp}.
\end{proof}

\section{Tree of shape morphisms}

Up to this section we have related categories of inverse sequences
with categories of simplicial trees and we have mentioned how this
can be used to describe a shape morphism as a map between trees.
In this last section we treat the spaces of shape morphisms
between compact connected metric spaces. We use the representation
of the shape morphisms as approximative maps since the spaces of
approximative maps can be given as the inverse limit of an inverse
sequence of maps. Thus, this inverse sequence corresponds to a
tree, the infinite branches will be the approximative maps (i.e.
the shape morphisms), and the ultrametric between these as end
space of a tree (\ref{end}) is equivalent, up to uniform
homeomorphism, to the ultrametric described by M. Morón and F. R.
Ruiz del Portal in \cite{M-P}.

\paragraph{Inverse limits and approximative maps}

Let $Y$ be a compactum in the Hilbert cube $I^\infty$, Borsuk
proves in \cite{Bo3} that there is
\[Y_1 \overset{p_1}{\leftarrow} Y_2 \overset{p_2}{\leftarrow} \ldots\] an inverse
system such that $\underset{\leftarrow}{\lim}Y_k =Y$ with
$Y_k\subset I^\infty$ prisms in the sense of Borsuk \cite{Bo3}
($Y_k$ is homeomorphic to the cartesian product $P\times
I^{\infty}$ with $P$ a compact polyhedron) such that $Y_k$ is a
neighborhood of $Y$, $Y_{k+1} \subset Y_{k}$ and $p_i$ the natural
inclusion. Let $X$ another compactum and $\{f_k\}_{k\in
\mathbb{N}}$ an approximative map of $X$ towards $Y$ in the sense
of Borsuk \cite{Bo1} with $f_k: X \to Y_k$.

\begin{displaymath}
\xymatrix{Y_1  & Y_2 \ar[l] & Y_3 \ar[l] & \cdots \ar[l]\\
            X \ar[u]|{f_1} \ar[ur]|{f_2} \ar[urr]|{f_3}& & &}
\end{displaymath}

\begin{prop}\label{aprox map} Given $\{f_k\}_{k\in \mathbb{N}}$ with $f_k:X \to
Y_k$ an approximative map then there exists $\{f'_k\}_{k\in
\mathbb{N}}$ with $f'_k:X \to Y_k$ an approximative map such that
$p_k\circ f'_{k+1}\simeq f'_k$ in $Y_k \quad \forall k\in
\mathbb{N}$ and $\{f_k\}_{k\in \mathbb{N}} \simeq \{f'_k\}_{k\in
\mathbb{N}}$.
\end{prop}

\begin{proof} By definition of approximative map we know that
\ $\forall N \ \exists m(N)$ such that $p_t \circ f_{t+1} \simeq
f_t$ in $Y_N \quad \forall t\geq m(N)$.

For $N_1=1$ there exists $m_1$ such that $p_t \circ f_{t+1}\simeq
f_t$ in $Y_1 \quad \forall t \geq m_1$. Define $g_{N_1}:=p_1 \circ
p_2 \circ \ldots \circ p_{m_1-1}\circ f_{m_1}: X \to Y_1$. Now let
$N_2=m_1$ and there exists $m_2$ such that $p_t\circ f_{t+1}\simeq
f_t$ in $Y_{N_2} \quad \forall t \geq m_2$. Then, define
$g_{N_2}:=p_{m_1} \circ p_{m_1+1} \circ \ldots \circ
p_{m_2-1}\circ f_{m_2}: X \to Y_{N_2}$. We can construct in this
way an inverse sequence $\{Y_{N_j}\}$
\[ Y_{N_1} \overset{p_{N_2N_1}}{\leftarrow}
Y_{N_2} \overset{p_{N_3N_2}}{\leftarrow} \ldots\] with
$p_{N_{i+1}N_i}$ the natural inclusion ($p_{N_{i+1}N_i}=p_{N_i}
\circ p_{N_i+1} \circ \ldots \circ p_{N_{i+1}}$) which is
equivalent to $\{Y_k\}_{k\in \mathbb{N}}$ since $\{N_j\}_{j\in
\mathbb{N}}$ is cofinal in $\mathbb{N}$.

Hence we have another approximative map from X towards Y,
$\{g_{N_i}\}_{i\in \mathbb{N}}$ with $g_{N_i}:=p_{m_{i-1}} \circ
p_{m_{i-1}+1} \circ \ldots \circ p_{m_i-1}\circ f_{m_i}: X \to
Y_{N_i}$. Clearly $g_{N_i} \simeq p_{N_{i+1}}^{N_i} \circ
g_{N_{i+1}}$ in $Y_{N_i} \quad \forall i$.

\begin{displaymath}
\xymatrix{Y_1  & Y_{N_2} \ar[l] & Y_{N_3} \ar[l] & \cdots \ar[l]\\
            X \ar[u]|{g_1} \ar[ur]|{g_{N_2}} \ar[urr]|{g_{N_3}}& & &}
\end{displaymath}

Now we can define the approximative map $\{g_i\}_{i\in
\mathbb{N}}$ with $g_i:=p_{N_i}^{i}\circ g_{N_i}:X \to Y_i \quad
\forall N_{i-1}<i<N_i$. It is quite easy to see that it represents
the same shape morphism. Following Borsuk's approximation, for any
neighborhood $V$ of $Y$ there exists $i_0$ such that
$Y_{N_i}\subset V \ \forall i\geq i_0$, and it is immediate to
check that $\{g_i\}_{i\in \mathbb{N}} \simeq \{g_{N_i}\}_{i\in
\mathbb{N}} \simeq \{f_i\}_{i\in \mathbb{N}}$.

Hence, for every shape morphism there exists a representative
which is an approximative map with the condition above.\end{proof}

Let $[X,Y_k]$ the homotopy classes of continuous maps from $X$ to
$Y_k$. Since $Y_k$ is a prism, we can prove that
$card([X,Y_k])\leq \aleph_0$. $p_k:Y_{k+1}\to Y_k$ induces a map
$p_k^*:[X,Y_{k+1}]\to [X,Y_k]$ and hence $([X,Y_k],p_k^*)$ is an
inverse sequence in \textbf{Tower-Set}. Clearly, an element in the
inverse limit is an approximative map. Then, in the correspondent
tree of this inverse sequence $(T_{X,Y},v)$, the geodesically
complete branches are given by sequences of vertices that
represent approximative maps.

\begin{prop} \label{biyeccion} There is a bijection between the homotopy
classes of approximative maps from X to Y and the geodesically
complete branches in $T_{X,Y}$.
\end{prop}

\begin{proof} Clearly a geodesically complete branch of the tree
represents an approximative map and by proposition \ref{aprox map}
each class of approximative maps is represented by a geodesically
complete branch in $T_{X,Y}$.\end{proof}

Let us recall that by $T_{X,Y}^\infty$ we denote the maximal
geodesically complete subtree of $T_{X,Y}$.

\begin{prop} Consider (Sh(X,Y),d) the space of shape morphisms defined
in \cite{M-P}. Then, $end(T_{X,Y}^\infty)$ is uniformly
homeomorphic to (Sh(X,Y),d).
\end{prop}

\begin{proof} It is well known the bijection between shape
morphisms and homotopy classes of approximative maps, see
\cite{MS1}. Hence, by \ref{biyeccion} we can assume this bijection
between shape morphisms and branches of $T_{(X,Y)}^\infty$.

$\forall \epsilon > 0$ there exists $n_0$ such that $Y_k\subset
B(Y,\frac{\epsilon}{2}) \quad \forall k\geq n_0$. Consider two
branches of $T_{(X,Y)}^\infty$ $F$ and $G$ such that
$\tilde{d}(F,G)<\delta=e^{-n_0}$ with the metric $\tilde{d}$ of
$end(T_{(X,Y)}^\infty)$. $F$ and $G$ represent two approximative
maps $\{f_k\}_{k\in \mathbb{N}}$ and $\{g_k\}_{k\in \mathbb{N}}$
such that $f_k\simeq g_k$ in $Y_k \quad \forall k\leq n_0$ and
since $p_k\circ f_{k+1} \simeq f_k$ in $Y_k \quad \forall k\in
\mathbb{N}$ we have that $f_k\simeq g_k$ in $Y_{n_0}$ and, in
particular in $B(Y,\frac{\epsilon}{2}) \quad \forall k \geq n_0$,
and hence for the respective shape morphisms
$\underline{f},\underline{g}, \quad
d(\underline{f},\underline{g})<\epsilon$.

On the other way, $\forall \epsilon > 0$ there exists $n_0$ such
that $e^{-n}<\epsilon \quad \forall n\geq n_0$, and since
$Y_{n_0}$ is a neighborhood of $Y$, there exists $\delta>0$ such
that $B(Y,2 \cdot \delta)\subset Y_{n_0}$. Consider two shape
morphisms (represented by two approximative maps)
$\underline{f},\underline{g}$ such that
$d(\underline{f},\underline{g})<\delta \Rightarrow \exists n_1$
such that $f_k\simeq g_k$ in $B(Y,2\cdot \delta)$, and in
particular in $Y_{n_0} \quad \forall k \geq n_1$, and since
$p_k^{n_0}\circ f_k \simeq p_k^{n_0}\circ g_k$ in $Y_{n_0}$ the
corresponding branches $F,G$ coincide at least on $[0,e^{-n_0}]$
and hence $\tilde{d}(F,G)<\epsilon$.\end{proof}

\begin{nota} Note that this result is independent from the
election of the sequence of prisms $Y_k$.
\end{nota}



We tried to see if this homeomorphism could hold some stronger
condition as being bi-Lipschitz or bi-Hölder and it doesn't.

\begin{ejp} Let $X=\{*\}$ a single point and
$Y=\{1,\frac{1}{2},\ldots,\frac{1}{2^n},\ldots,0\}$.
\end{ejp}
The shape morphisms are represented by the maps
\[Sh(X,Y):=\left\{
\begin{tabular}{l} $\alpha_n \quad \mbox{ such that } \alpha_n(*)=\{\frac{1}{2^n}\},$\\

$\alpha_0 \quad \mbox{ such that } \alpha_0(*)=\{0\}$\end{tabular}
 \right.\]

Clearly $d(\alpha_0,\alpha_n)=\frac{1}{2^{n+1}}$ and
$d(\alpha_n,\alpha_{n+1})=\frac{1}{2^{n+2}}$ in $(Sh(X,Y),d)$.

Now we can choose an inverse system of compact neighborhoods
$\{Y_k\}_{k\in \mathbb{N}}$ with $Y_k \subset Y_{k+1}$ and
$p_k:Y_{k+1}\to Y_k$ the natural inclusion such that

$\alpha_i\simeq \alpha_j$ ( this is $\alpha_i(*) \mbox{ and }
\alpha_j(*)$ are in the same path-component) in  $Y_1,Y_2, \ldots
Y_{n_1} \quad \forall i,j \in \mathbb{N}\cup \{0\}$, with
$n_1>-ln\Big(\frac{1}{4}\Big)$ and

$\alpha_i\simeq \alpha_j$ in $Y_{n_{k-1}+1}, \ldots Y_{n_k} \quad
\forall i,j\geq k-1$, with $n_k>-k \cdot
ln\Big(\frac{(\frac{1}{2^{k+1}})}{k}\Big) \linebreak \forall k\geq
2$.

In this case it is clear that
$\tilde{d}(\alpha_{k-1},\alpha_k)=e^{-n_k}<\Big(\frac{(\frac{1}{2^{k+1}})}{k}\Big)^{k}
=\Big(\frac{d(\alpha_{k-1},\alpha_k)}{k}\Big)^{k}$. Thus, for any
constant $C>0$ and $0<l<1$ there exists $k_0$ such that $\forall
k>k_0$\linebreak
$C\cdot(\tilde{d}(\alpha_{k-1},\alpha_k))^l<C\cdot
(\tilde{d}(\alpha_{k-1},\alpha_k))^{\frac{1}{k}} <k\cdot
(\tilde{d}(\alpha_{k-1},\alpha_k))^{\frac{1}{k}}<d(\alpha_{k-1},\alpha_k)$
and hence, the uniform homeomorphism is not bi-Hölder.

\vspace{0.5cm}

Using these trees of shape morphisms we are able to obtain the
next result from \cite{M-P} about how composition induces
uniformly continuous maps between the spaces of shape morphisms.

\begin{prop} Let $X,Y,Z$ be compact metric spaces and $F:X \to Y$ a
\emph{shape} morphism. If we build, using inverse sequences of
neighborhoods totally ordered by inclusion with inverse limits $X$
and $Y$, $T_{Z,X}$ and $T_{Z,Y}$, and define
$F_*:end(T^\infty_{Z,X})\to end(T^\infty_{Z,Y})$ as $F_*(\alpha)=F
\circ \alpha$, then $F_*$ is uniformly continuous.
\end{prop}

\begin{proof} Let $\underline{X}=X_1 \leftarrow X_2 \leftarrow \cdots$,
$\underline{Y}=Y_1 \leftarrow Y_2 \leftarrow \cdots$ and
$\underline{Z}=Z_1 \leftarrow Z_2 \leftarrow \cdots$ inverse
sequences of neighborhoods connected by inclusions such that
$X=\underset{\leftarrow}{lim} X_i$, $Y=\underset{\leftarrow}{lim}
Y_i$ and $Z=\underset{\leftarrow}{lim} Z_i$. Let $F\in Sh(X,Y)$.
Then $F$ will be represented by an approximative map
$\underline{f}: X \to \underline{Y}$. Let us see that $F_*$
induces a morphism of inverse sequences between $([Z,X_k],i^*_k)$
and $([Z,Y_k],i^*_k)$. Given $\underline{f}: X \to \underline{Y}$,
see Lemma 1, page 333 in \cite{MS1}, there exists a fundamental
sequence $(\Phi_n): X \to Y$ such that for every $k\in
\mathbb{N}$, $\Phi_k|_X=f_k$ and $\Phi_{k'}|_{U_k} \simeq
\Phi_k|_{U_k} $ in $Y_k$, $k'\geq k$ for some neighborhood $U_k$
of $X$. In particular, $\Phi_k(U_k)\subset Y_k$ and there exists
some level $m_k$ for which $X_{m_k}\subset U_k$. Then, the map
$\Phi_{k*}: [Z,X_{m_k}] \to [Z,Y_k]$ given by $\Phi_{k *}(f_k)=
\Phi_k \circ f_k$ is well defined. We can assume that $(m_k)$ is
increasing and to check that this induces a morphism between
inverse sequences it suffices to see that the following diagram
commutes:

\begin{displaymath}
\xymatrix{[Z,X_{m_k}] \ar[d]_{\Phi_{k*}}  & [Z,X_{m_{k+1}}] \ar[l]_{i_*}  \ar[d]^{\Phi_{k+1*}}\\
            [Z,Y_k] & [Z,Y_{k+1}] \ar[l]^{i_*}}
\end{displaymath}

Let $[f_{m_{k+1}}]\in [Z,X_{m_{k+1}}]$ and consider $i_* \circ
\Phi_{k+1} \circ f_{m_{k+1}}:Z \to Y_k$. From the definition of
$\Phi_k$ we know that $\Phi_k|_{X_{m_k}}\simeq \Phi_{k+1}
|_{X_{m_{k}}}$ in $Y_k$, therefore $i_* \circ \Phi_{k+1} \circ
f_{m_{k+1}}\simeq \Phi_{k} \circ i_* \circ f_{m_{k+1}}:Z \to Y_k$
and the diagram commutes.

A morphism between inverse sequences induces, see \ref{equiv
categ}, a rooted continuous metrically proper map between the
trees which may be restricted to a map with the same properties
between the maximal geodesically complete subtrees. This map, can
be translated with \ref{tma equiv} to a uniformly continuous map
between the end spaces, and those are the spaces of shape
morphisms with their ultrametrics (with those depending, up to
uniform homeomorphism, on the inverse sequences initially chosen).
\end{proof}

\paragraph{Inverse limits and Marde\v{s}i\'c-Segal's approach to shape
morphisms} Let X,Y two compacta. Marde\v{s}i\'c and Segal proved
in \cite{MS1} $\S 5.2$, see also \cite{MS2}, that there are
inverse sequences in the homotopy category $\mathcal{P}$ of
topological spaces having the homotopy type of polyhedra
$\textbf{X}:=X_1 \overset{p_1}{\leftarrow} X_2
\overset{p_2}{\leftarrow} \cdots $ and $\textbf{Y}:=Y_1
\overset{q_1}{\leftarrow} Y_2 \overset{q_2}{\leftarrow} \cdots $
such that $X=\underset{\leftarrow}{lim}X_i$,
$Y=\underset{\leftarrow}{lim}Y_i$ and $\textbf{p}:X \to
\textbf{X}$, $\textbf{q}:Y \to \textbf{Y}$
$\mathcal{P}$-expansions. They also defined the shape morphisms
between $X$ and $Y$ as homotopy classes of morphisms in
pro-$\mathcal{P}$ between \textbf{X} and \textbf{Y} and proved
that those morphism can be given by homotopy classes of morphism
in \textbf{pro-Top}, with \textbf{Top} the category of topological
spaces, between $X$ and \textbf{Y}. They also proved that if we
restrict ourselves to the Hilbert cube, there is an isomorphism of
categories between this category and Borsuk's Shape category.

Homotopy classes of morphism in \textbf{pro-Top} between $X$ and
\textbf{Y} can be given as inverse limits of the inverse sequence
$([X,Y_k],q_k*)$. Thus, if we consider $T_{X,Y}$ the tree of this
inverse sequence, we have the following proposition. (Obviously,
it may be given as a corollary of \ref{biyeccion} but it seems
interesting to include here a direct proof of this).

\begin{prop}\label{biy} There is a bijection between the shape morphisms
of X to Y and the set of geodesically complete branches in
$T_{X,Y}$.
\end{prop}

\begin{proof} First we define a function $\xi$ from the geodesically
complete branches of the tree to the shape morphisms. A
geodesically complete branch of the tree obviously represents a
morphism $\textbf{f}:X \to \textbf{Y}$, in \textbf{pro-HTop}
(where \textbf{HTop} is the homotopy category of topological
spaces), which is a commutative diagram as follows.

\begin{displaymath}
\xymatrix{Y_1  & Y_2 \ar[l] & Y_3 \ar[l] & \cdots \ar[l]\\
            X \ar[u]|{f_1} \ar[ur]|{f_2} \ar[urr]|{f_3}& & &}
\end{displaymath}

Since $Y_k$ is in $\mathcal{P}$, let $p:X \to \textbf{X}$ be any
$\mathcal{P}$-expansion of $X$, see \cite{MS1}. Thus, for any
morphism $\textbf{f}:X \to \textbf{Y}$ in pro-$\mathcal{T}$ there
exists a unique morphism $\textbf{h}: \textbf{X}\to \textbf{Y}$ in
pro-$\mathcal{P}$ making commutative the diagram.

\begin{displaymath}
\xymatrix{\textbf{X} \ar[dr]_{\textbf{h}} & X \ar[l]_p \ar[d]^{\textbf{f}} \\
              & \textbf{Y}   }
\end{displaymath}

This means that for any morphism $\textbf{f}:X \to \textbf{Y}$ in
\textbf{pro-HTop}, this is any geodesically complete branch $F$ of
the tree, there is a unique homotopy class $\textbf{[h]}$ of
morphisms in pro-$\mathcal{P}$ making the diagram commutative,
this is, a unique shape morphism $H:X\to Y$. So we define
$\xi(F)=H$.

$\xi$ is injective. Let $F,F'$ be infinite branches and
$\textbf{f},\textbf{f'}:X \to \textbf{Y}$ the corresponding
morphisms in pro-$\mathcal{T}$ and suppose that
$\xi(\textbf{f})=H=\textbf{[h]}$ and
$\xi(\textbf{f'})=H'=\textbf{[h']}$ are such that $\textbf{h}\sim
\textbf{h'}$. This means that $\forall n\in \mathbb{N}$ there
exists some $m \in \mathbb{N}$, $m\geq \Phi(n),\Phi'(n)$, such
that the diagram commutes:

\begin{displaymath}
\xymatrix{X_{\Phi_m} \ar[dr]_{h_n} & X_m \ar[l] \ar[r] & X_{\Phi'_m} \ar[dl]^{h'_n}\\
              & Y_n  & }
\end{displaymath}

Clearly $h_n \circ p_{\Phi(n)m}\simeq h'_n \circ p_{\Phi'(n)m}$
implies that if we compose with $p_m:X\to X_m$ of the
$\mathcal{P}$-expansion $\textbf{p}$ we have that,
\begin{equation}h_n \circ p_{\Phi(n)m} \circ p_{m}\simeq h'_n \circ
p_{\Phi'(n)m}\circ p_m.
\end{equation}
Since $\textbf{p}$ is a morphism in pro-$\mathcal{T}$
$p_{\Phi(n)m}\circ p_m\simeq p_{\Phi_n}$ and $p_{\Phi'(n)m}\circ
p_m\simeq p_{\Phi'_n}$ and by definition, $\textbf{h}\circ
\textbf{p}\simeq \textbf{f}$, this is, $\forall n\in \mathbb{N}$,
$h_n\circ p_{\Phi(n)}\simeq f_n$ and $h'_n\circ p_{\Phi(n)}\simeq
f'_n$. Then we have that $\forall n \in \mathbb{N}$
\begin{equation}f_n\simeq h_n\circ p_{\Phi(n)m} \circ
p_{m}\simeq h'_n \circ p_{\Phi'(n)m}\circ p_m \simeq f'_n.
\end{equation}

Hence $\textbf{f}\sim \textbf{f'}$ and $F=F'$.

$\xi$ is surjective. Consider any shape morphism between X and Y
given by a morphism in pro-$\mathcal{P}$ between the inverse
sequences, $h:\textbf{X}\to \textbf{Y}$. Then if we consider
$\textbf{f}:X \to \textbf{Y}$ defined by $f_k:=p_{\Phi(k)}\circ
h_{\Phi(k)}:X\to Y_k$ and $F$ the corresponding branch then
obviously $\textbf{f}\sim \textbf{h} \circ \textbf{p}$, and the
uniqueness of $\textbf{[h]}$ in the $\mathcal{P}$-expansion
implies that $H=\textbf{[h]}=\xi(F)$. \end{proof}

\paragraph{Pointed shape.} Let $(X,*),(Y,*)$ two pointed metric compacta, then if $\mathcal{P}_*$
is the category of spaces with the (pointed) homotopy type of
pointed polyhedra, there are also defined in \cite{MS1} pointed
shape morphisms as (pointed) homotopy classes of morphisms in
pro-$\mathcal{P}_*$.

We can now define in a similar way a tree $T_{X*,Y*}$ whose
vertices are pointed homotopy classes of maps from $(X,*)$ to
$(Y_n,*)$ (denoted $[(X,*),(Y_n,*)]$) $\forall n\in \mathbb{N}$
and joining them in a similar way. There is an edge joining
$[\alpha]\in [(X,*),(Y_{k+1},*)]$ and $[\beta]\in [(X,*),(Y_k,*)]$
if and only if $[p_k\circ \alpha]\simeq_* [\beta]$ in $(Y_k,*)$. A
proof similar to the one given in the non-pointed case
establishes:

\begin{prop} There is a bijection between the pointed shape
morphisms of $(X,*)$ to $(Y,*)$ and the set of geodesically
complete branches in $T_{X*,Y*}$.
\end{prop}

If we consider the first shape group, the (pointed) morphisms from
$(S^1,*)$ to $(Y,*)$ may be considered geodesically complete
branches of the tree defined over the inverse system
$\textbf{Y}_{*}:=(Y_1,*) \overset{q^*_1}{\leftarrow} (Y_2,*)
\overset{q^*_2}{\leftarrow} \cdots $.

Now, as an example of this geometric point of view, let us analyze
the solenoid. It is well known that the first shape group of the
solenoid is trivial. Let us recall here the construction.

\begin{ejp} Consider a solenoid $(Y,z_0)$ which is the inverse limit of the
following inverse system in pro-$\mathcal{P_*}$.
$(Y_n,z_0)=(S^1,z_0) \ \forall n\in \mathbb{N}$ (with $S^1:=\{z\in
\mathbb{C} \mbox{ with } ||z||=1\}$ and $z_0=1$) and the bonding
(pointed) maps $p_n: (Y_{n+1},z_0) \to (Y_n,z_0)$ are defined by
$p(z)=z^2 \quad \forall n\in \mathbb{N}$.
\end{ejp}

Each level of vertices of the tree, $[(S^1,z_0),(Y_n,z_0)]$, has
structure of group. It is in fact the first homotopy group of
$(Y_n,z_0)$ which is isomorphic to $(\mathbb{Z},+)$ (let
$h_n:[(S^1,z_0),(Y_n,z_0)] \to (\mathbb{Z},+)$ be this
isomorphism), and the bonding maps $p_n$ clearly induce
endomorphisms $f_n$ in $(\mathbb{Z},+)$ such that $f_n(1)=2$ and
hence $f_n(z)=2\cdot z$.

This implies immediately that the first shape group of the
solenoid is trivial. If we consider the tree, $T_S$, associated to
this inverse sequence, the trivial pointed shape morphism is
represented by the geodesically complete branch whose vertex in
each $[(S^1,z_0),(Y_n,z_0)]$ is the trivial map $f(z)=z_0$
($h_n(f)=0$ in $(\mathbb{Z},+)$).

Any geodesically complete branch of the tree representing a
non-trivial pointed shape morphism from $(S^1,z_0)$ to $(Y,z_0)$
would be determined by a sequence of vertices $\alpha_n \in
[(S^1,*),(Y_n,*)]$ which can be identified with a sequence of
integers $(z_1,z_2,z_3,\ldots)$ with $0\neq z_n=h_n(\alpha_n)$.
The bonding maps impose the condition that
$z_n=f_n(z_{n+1})=2\cdot z_{n+1}$ but this leads to a
contradiction. There must be some $k\in \mathbb{N}$ such that
$2^k$ doesn't divide $z_1$ and this contradicts the fact that
$z_1=f_1 \circ f_2\circ \ldots \circ f_k (z_{k+1})=2^k\cdot
z_{k+1}$. Thus, the maximal geodesically complete subtree consists
of a unique infinite branch.

Nevertheless, there are arbitrarily long branches in the tree
$T_S$, which means that the tree is not metrically proper homotopy
equivalent to the maximal geodesically complete subtree. This
corresponds, as we saw in \ref{ML-prop}, to the sequence not being
(ML), which is one of the basic properties of this sequence since
the solenoid is not movable.

\begin{ejp} The same works for any solenoid defined with
bonding (pointed) maps $p_n: (Y_{n+1},z_0) \to (Y_n,z_0)$ defined
by $p(z)=z^{p_n}$ with $p_n$ prime $\forall n\in \mathbb{N}$.
\end{ejp}

In this case the induced endomorphisms are such that $f_k(1)=p_k$
and so $f_k(z)=p_k \cdot z$. Any geodesically complete branch $F$
is represented by a sequence of integers $(z_1,z_2,z_3,\ldots)$
with $0\neq z_n=h_n(\alpha_n)$ and the bonding maps impose the
condition that $z_n=f_n(z_{n+1})=p_n\cdot z_{n+1}$. Let
$z_1=p_1\cdot z_2=p_2\cdot p_1 \cdot z_3= \ldots$ and since $z_1$
is a finite product of primes there must be some $k\in \mathbb{N}$
such that $z_k=1$ and this contradicts the fact that $z_k=p_k\cdot
z_{k+1}$.

\end{document}